\documentclass[12pt, reqno]{amsart}
\usepackage{graphicx} % to input postscript files
\numberwithin{equation}{section}

\newtheorem{remark}{Remark}

\newtheorem{theorem}{Theorem}
\newtheorem{cor}{Corollary}

\begin{document}

\baselineskip 18pt

\newcommand{\E}{\mathbb{E}}
\newcommand{\Eof}[1]{\mathbb{E}\left[ #1 \right]}
\newcommand{\Pof}[1]{\mathbb{P}\left[ #1 \right]}
\renewcommand{\H}{\mathbb{H}}
\newcommand{\R}{\mathbb{R}}
\newcommand{\sigl}{\sigma_L}
\newcommand{\BS}{\rm BS}
\newcommand{\p}{\partial}
\renewcommand{\P}{\mathbb{P}}
\newcommand{\var}{{\rm var}}
\newcommand{\cov}{{\rm cov}}
\newcommand{\beaa}{\begin{eqnarray*}}
\newcommand{\eeaa}{\end{eqnarray*}}
\newcommand{\bea}{\begin{eqnarray}}
\newcommand{\eea}{\end{eqnarray}}
\newcommand{\ben}{\begin{enumerate}}
\newcommand{\een}{\end{enumerate}}
\newcommand{\bit}{\begin{itemize}}
\newcommand{\eit}{\end{itemize}}

\newcommand{\bX}{\boldsymbol{X}}
\newcommand{\bx}{\mathbf{x}}
\newcommand{\bE}{\mathbf{e}}
\newcommand{\bw}{\mathbf{w}}
\newcommand{\bW}{\boldsymbol{W}}
\newcommand{\bB}{\boldsymbol{B}}
\newcommand{\bZ}{\boldsymbol{Z}}
\newcommand{\bH}{\mathbf{H}}
\newcommand{\bF}{\mathbf{F}}
\newcommand{\bG}{\mathbf{G}}
\newcommand{\bs}{\mathbf{s}}
\newcommand{\bsihat}{\widehat{\bs_i}}

\newcommand{\cL}{\mathcal{L}}

\newcommand{\mt}{\mathbf{t}}
\newcommand{\mS}{\mathcal{S}}

\newcommand{\argmax}{{\rm argmax}}
\newcommand{\argmin}{{\rm argmin}}

\newcommand{\tM}{\widetilde{M}}
\newcommand{\tE}{\tilde{\mathbb{E}}}
\newcommand{\tEof}[1]{\tilde{\mathbb{E}}\left[ #1 \right]}
\newcommand{\tP}{\tilde{\mathbb{P}}}
\newcommand{\tW}{\tilde{W}}
\newcommand{\1}{\mathbf{1}}
\renewcommand{\O}{\mathcal{O}}
\newcommand{\dt}{\Delta t}
\newcommand{\tr}{{\rm tr}}

\newcommand{\Xv}{X^{(v)}}
\newcommand{\Xvs}{X^{(v^*)}}
\newcommand{\Jv}{J^{(v)}}

\newcommand{\cG}{\mathcal{G}}
\newcommand{\cF}{\mathcal{F}}
\newcommand{\cLv}{\mathcal{L}^{(v)}}

\allowdisplaybreaks

%%***************************************************************************
%%
%%  Document begins here
%%
%%***************************************************************************

\title[Bessel bridge representation hyperbolic heat kernel]
{Bessel bridge representation for heat kernel in hyperbolic space}
\begin{abstract}
This article shows a Bessel bridge representation for the transition density of Brownian motion on the Poincar\'e space. This transition density is also referred to as the heat kernel on the hyperbolic space in differential geometry literature. The representation recovers the well-known closed form expression for the heat kernel on hyperbolic space in dimension three. However, the newly derived bridge representation is different from the McKean kernel in dimension two and from the Gruet's formula in higher dimensions. The methodology is also applicable to the derivation of an analogous Bessel bridge representations for heat kernel on a Cartan-Hadamard radially symmetric space and for the transition density of hyperbolic Bessel process.
\end{abstract}

\author[X. Cheng and T.-H. Wang]
{Xue Cheng and Tai-Ho Wang}

\address{Xue Cheng \newline
LMEQF, Department of Financial Mathematics \newline
School of Mathematical Sciences, Peking University \newline
Beijing, China
}
\email{chengxue@pku.edu.cn}

\address{Tai-Ho Wang \newline
Department of Mathematics \newline
Baruch College, The City University of New York \newline
1 Bernard Baruch Way, New York, NY10010
}
\email{tai-ho.wang@baruch.cuny.edu}

\maketitle

%%%%%%%%%%%%%%%%%%%%%%%%%%%%%%%%%%%%%%%%%%%%%%%%%%%%%%%%%%%%%%%%%%%%%%%%%%%%%%%%%%%%%%%%
%
%  Section: Introduction
%
%%%%%%%%%%%%%%%%%%%%%%%%%%%%%%%%%%%%%%%%%%%%%%%%%%%%%%%%%%%%%%%%%%%%%%%%%%%%%%%%%%%%%%%%

\section{Introduction}

Heat kernel, also known as the fundamental solution for heat operator, plays a crucial role in various branches of mathematics including analysis, differential geometry, and probability theory.
On Euclidean spaces, heat kernels have closed form expression given by the Gaussian kernels, which also serve as the transition density of Euclidean Brownian motions. Deriving, to some extent, analytical expression of heat kernel on general curved space is more involved if not completely impossible. Symmetry of the underlying space plays an important role. Hyperbolic space is one of the symmetry spaces with constant negative curvature that has an expression for heat kernel in analytic form. As we shall see throughout the article, due to symmetry, expressions for heat kernels on hyperbolic space depend solely on geodesic distance. 

Derivations of heat kernel on hyperbolic space in closed or quasi closed forms have been done by various authors. We list notable a few as follows. McKean in \cite{mckean} presented a quasi closed form expression (up to an integral), nowadays known as the McKean kernel, for heat kernel on two dimensional hyperbolic space, see \eqref{eqn:mckean} below. A detailed derivation of the McKean kernel using Fourier transform, isometries, and eigenvalues and eigenfunctions of the Laplace-Beltrami operator can be found in \cite{chavel} (Section 2 in Chapter X). The closed form expression for heat kernel on three dimensional hyperbolic space, see \eqref{eqn:hyp-ht-kern-3d} below, and the Millson's recursion formula for higher dimensional hyperbolic heat kernel were reported in \cite{debiard-gaveau-mazet}. A different proof of Millson's recursion formula based on the relationship between the heat kernel and the wave kernel and the explicit formula for wave kernel on symmetry space was given in \cite{grigoryan}. The following expression obtained in \cite{gruet} for the $n$-dimensional hyperbolic heat kernel is known as the Gruet's formula 
\beaa
p_{\H^n}(t,z,w) = \frac{e^{-(n-1)^2t/8}}{\pi (2\pi)^{n/2}\sqrt t} \Gamma\left( \frac{n+1}2 \right) \int_0^\infty \frac{e^{(\pi^2 - b^2)/2t} \sinh(b) \sin(\pi b/t)}{[\cosh(b) + \cosh(r)]^{(n+1)/2}} db,
\eeaa
where $r = d(z,w)$ is the geodesic distance between $z, w \in \H^n$. We refer the interested reader to \cite{matsumoto-yor} for a derivation of the Gruet's formula and its relationship to the pricing of Asian options. A probabilistic approach, which is different from the one employed in the current article, of deriving the heat kernel on two dimensional hyperbolic space can also be found in \cite{ikeda-matsumoto}. As closed form expression is concerned, \cite{matsumoto} obtained  expressions for heat kernels on symmetric spaces of rank 1. Finally, it is worth  mentioning that a nice application of the McKean kernel in quantitative finance can be found in \cite{hlw}.

In this article, we prove yet another representation for heat kernel on hyperbolic space: the Bessel bridge representation in Theorem \ref{thm:bessel-bridge}. By working under geodesic polar coordinates, Brownian motion in hyperbolic space is decomposed into a one dimensional process in the radial part and a process on the unit sphere of codimension one. The radial part, also known as the hyperbolic Bessel process, is indeed a Brownian motion with drift. Due to symmetry of hyperbolic space, the drift in the radial part depends only on the geodesic distance. Girsanov's theorem allows us to define an equivalent probability measure on the underlying probability space through a Radon-Nikodym derivative so that, in the new probability measure, the radial process becomes a Bessel process of order $n$. We then substitute the stochastic integral that results from the Radon-Nikodym derivative with a Riemann integral by applying Ito's formula. The bridge representation of hyperbolic heat kernel is thus obtained by conditioning on the terminal point of the radial process in the new probability measure. The whole procedure is implemented in the proof of Theorem \ref{thm:bessel-bridge}. Similar representation for the transition density of hyperbolic Bessel process is shown in Theorem \ref{thm:hyperbolic-bessel}. With minor modifications, the same procedure is also applicable to the case of general Cartan-Hadamard radially symmetric spaces and the result is summarized in Theorem \ref{thm:bessel-bridge-radial-sym}. We remark that, in the one dimensional case, the idea of bridge representation for transition density of a diffusion first appeared, to our knowledge, in \cite{rogers}, see also \cite{w-gatheral} for further discussions. 

%%%%%%%%%%%%%%%%%%%%%%%%%%%%%%%%%%%%%%%%%%%%%%%%%%%%%%%%%%%%%%%%%%%%%%%%%%%%%%%%%%%%%%%%
%
%  Section: Heat kernel on hyperbolic space
%
%%%%%%%%%%%%%%%%%%%%%%%%%%%%%%%%%%%%%%%%%%%%%%%%%%%%%%%%%%%%%%%%%%%%%%%%%%%%%%%%%%%%%%%%

\section{Heat kernel on hyperbolic space}
Throughout the text, stochastic processes and random variables are assumed defined on the complete filtered probability space $(\Omega,\cF, \P, \{\cF_t\}_{t\in[0,\infty)})$ satisfying the usual conditions. We shall denote the $n$-dimensional hyperbolic space by $\H^n$ and the associated heat kernel on $\H^n$ between $z, w \in \H^n$ at time $t$ by $p_{\H^n}(t,z,w)$. 

%%%%%%%%%%%%%%%%%%%%%%%%%%%%%%%%%%%%%%%%%%%%%%%%%%%%%%%%%%%%%%%%%%%%%%%%%%%%%%%%%%%%%%%%
%  Subsection: The hyperbolic space
%%%%%%%%%%%%%%%%%%%%%%%%%%%%%%%%%%%%%%%%%%%%%%%%%%%%%%%%%%%%%%%%%%%%%%%%%%%%%%%%%%%%%%%%
\subsection{The hyperbolic space}
Conventionally, hyperbolic spaces are parametrized by two isometrically equivalently models: the half-space model and the ball model. 
The underlying space in the half-space model is the half plane $\mathbb{R}^n_+ = \{(x_1, \cdots, x_n) : x_n > 0 \}$, while the underlying space in the ball model is the open ball $\mathbb{B}_n = \{ y = (y_1, \cdots, y_n) : \|y\| < 1 \}$. The transformation between the two models can be found for instance in \cite{chavel} (p.264). In particular, for $n=2$, the transformation between the two models is given by the M\"obius transform, $T: \mathbb{C}_+ \to \mathbb{B}_2$,
\[
  w = T(z) = \frac{z- i}{z+i}.
\]
In the half-space model, the metric $ds^2$ and the Laplace-Beltrami are given respectively by 
\beaa
  && ds^2 = \frac{dx_1^2 + \cdots + dx_n^2}{x_n^2}, \\
  && \Delta_M = x_n^2 \left( \p_{x_1}^2 + \cdots + \p_{x_n}^2 \right) + (2-n) x_n \p_n;
\eeaa
whereas in the ball model, the metric is given, in polar coordinates $(\rho, \theta)$, $\theta \in S^{n-1}$, by 
\beaa
  && ds^2 = \frac4{(1 - \rho^2)^2} \left( d\rho^2 + \rho^2 d\theta^2 \right), \\
  && \Delta_M = \frac{(1 - \rho^2)^2}4 \left( \p_\rho^2 + \frac1\rho \p_\rho + \frac1{\rho^2} \Delta_{S^{n-1}} \right),
\eeaa
where $d\theta^2$ is the Riemann metric and $\Delta_{S^{n-1}}$ the Laplace-Beltrami operator on the standard unit sphere $S^{n-1}$.
Moreover, if we make the transformation $\rho = \tanh\left( \frac r2 \right)$, then $(r,\theta)$ becomes the geodesic polar coordinates for $\H^n$.
We shall be working primarily in the geodesic polar coordinates $(r,\theta) \in [0,\infty) \times S^{n-1}$ under which the Riemann metric and the Laplace-Beltrami operator of $\H^n$ are given respectively as
\beaa
  && ds^2 = dr^2 + \sinh^2 r d\theta^2, \\
  && \Delta_{\H^n} = \p_r^2 + (n-1) \coth r \p_r + \frac1{\sinh^2r} \Delta_{S^{n-1}}.
\eeaa
We remark that the geodesic polar coordinates on $\H^n$ is a global diffeomorphism, henceforth defines a global coordinates, since hyperbolic space is Cartan-Hadamard.

%%%%%%%%%%%%%%%%%%%%%%%%%%%%%%%%%%%%%%%%%%%%%%%%%%%%%%%%%%%%%%%%%%%%%%%%%%%%%%%%%%%%%%%%
%  Subsection: The heat kernel
%%%%%%%%%%%%%%%%%%%%%%%%%%%%%%%%%%%%%%%%%%%%%%%%%%%%%%%%%%%%%%%%%%%%%%%%%%%%%%%%%%%%%%%%
\subsection{The heat kernel}
Generally speaking, heat kernel on a differentiable manifold $M$ is a fundamental solution to the (probabilist's) heat operator $\p_t - \frac12 \Delta_M$, where $\Delta_M$ is the Laplace-Beltrami operator on $M$. The minimal heat kernel also serves as the transition density of Brownian motion on $M$. 
We refer the reader to \cite{eltonbook} for expositions of Brownian motions on manifolds and their relationships to heat kernel. 
For reader's reference, we reproduce the heat kernel on the two and three dimensional hyperbolic spaces in the following.
\bea
 p_{\H^2}(z,w,t) = \frac{\sqrt2 e^{-t/8}}{(2\pi t)^{3/2}} \int_{d(z,w)}^\infty \frac{\xi e^{-\frac{\xi^2}{2t}}}{\sqrt{\cosh\xi - \cosh d(z,w)}} d\xi \label{eqn:mckean}
\eea
and
\begin{equation}
 p_{\H^3}(z,w,t) = \frac{e^{-\frac{r^2}{2t}}}{(2\pi t)^{3/2}} e^{-\frac{t}2} \frac r{\sinh(r)}, \label{eqn:hyp-ht-kern-3d}
\end{equation}
where $r = d(z,w)$ is the geodesic distance between $z$ and $w$. Note that the heat kernels given in \eqref{eqn:mckean} and \eqref{eqn:hyp-ht-kern-3d} are densities with respect to the volume form.

In the following, we apply Girsanov's theorem to derive an expression for the heat kernel over $\H^n$, for $n \geq 2$, in which the closed form expression (\ref{eqn:hyp-ht-kern-3d}) for $\H^3$ is recovered. 
Note that since the Laplace-Beltrami operator on $\H^n$ is rotationally invariant, the heat kernel, or equivalently the transition density for the Brownian motion on hyperbolic space, is also rotationally invariant, hence a radial function.

Consider the processes $(R_t, \Theta_t)$ governed by the SDEs
\bea
  && dR_t = dW_t + \frac{n-1}2 \coth(R_t)\, dt, \label{eqn: R-process} \\
  && d\Theta_t = \frac1{\sinh(R_t)} dZ_t \label{eqn: theta-process},
\eea
where $W_t$ is a standard one dimensional Brownian motion and $Z_t$ is a Brownian motion on the standard sphere $S^{n-1}$, independent of $W_t$. The infinitesimal generator of the process $(R_t,\Theta_t)$ is $\frac12 \Delta_{\H^n}$. Thus, it represents a Brownian motion on $\H^n$ in geodesic polar coordinates. We set the initial condition $\Theta_0$ to be a random variable uniformly distributed on $S^{n-1}$ so that the distribution of  $\Theta_t$ remains uniformly distributed on $S^{n-1}$ for all $t$. The main result of the article is given in the following theorem. 

\begin{theorem}(Bessel bridge representation) \label{thm:bessel-bridge} \\ 
Let $z, w \in \H^n$. The heat kernel $p_{\H^n}(T,z,w)$ on the hyperbolic space $\H^n$ has the following representation
\bea
p_{\H^n}(T,z,w) = e^{- \frac{(n-1)^2T}8} \left(\frac{r}{\sinh r} \right)^{\frac{n-1}2} \frac{e^{-\frac{r^2}{2T}}}{(2\pi T)^{\frac n2}} \tE_{r} \left[ e^{- \frac{(n-1)(n-3)}8 \int_0^T \left[\frac1{\sinh^2(R_t)} - \frac1{R_t^2} \right] dt} \right], \label{eqn:bessel-bridge}
\eea
where $r = r(z,w)$ is the geodesic distance between $z$ and $w$. $\tE_r[\cdot]$ denotes the conditional expectation $\tE[\cdot|R_T=r]$, where $R_t$ is a Bessel process of order $n$ in the $\tP$-measure. 
\end{theorem}
\begin{proof}
Let $(R_t, \Theta_t)$ be the process satisfying \eqref{eqn: R-process}:\eqref{eqn: theta-process} with initial conditions $R_0=0$ and $\Theta_0$ being uniformly distributed on $S^{n-1}$. 
We start with calculating the expectation of an arbitrary bounded measurable radial function $f$ as 
\beaa
  && \E[f(R_T)] = \int_{S^{n-1}}\int_0^\infty f(r) p(T,r) \sinh^{n-1} r dr d\omega,
\eeaa
where $d\omega$ is the volume form on $S^{n-1}$ and $p(T,r) = p_{\H^n}(T,z,w)$, $r = r(z,w)$ denotes the geodesic distance between $z$ and $w$. For the radial process $R_t$,
define the new measure $\tP$ by the Radon-Nikodym derivative
\begin{equation}
  \frac{d\tP}{d\P} = e^{\int_0^T h(R_t) dW_t - \frac12 \int_0^T h^2(R_t)dt}, \label{eqn:RN-r}
\end{equation}
where $h(r) = \frac{n-1}2\left[\frac1r - \coth(r)\right]$. Note that $h$ is a bounded function, in fact, $|h(r)|\leq \frac{n-1}2$ for all $r\geq 0$. Thus (\ref{eqn:RN-r}) is a well-defined change of probability measure. 
Therefore, Girsanov's theorem implies that, under the measure $\tP$, $W_t$ is a Brownian motion with drift $h$.
Moreover, in the $\tP$-measure, the SDE for the radial process $R_t$ becomes
\[
  dR_t = d\tW_t + \frac{n-1}2\, \frac{dt}{R_t},
\]
which is a Bessel process of order $n$. Therefore, we have
\bea
  && \E[f(R_T)] = \tE\left[ f(R_T) \frac{d\P}{d\tP} \right]
  = \tE\left[ f(R_T) e^{-\int_0^T h(R_t) dW_t + \frac12 \int_0^T h^2(R_t)dt} \right] \nonumber \\
  &=& \tE\left[ f(R_T) e^{-\int_0^T h(R_t) d\tW_t - \frac12 \int_0^T h^2(R_t)dt} \right]. \label{eqn:f-exp-tilde}
\eea
We substitute the stochastic integral in \eqref{eqn:f-exp-tilde} with a Riemann integral by applying Ito's formula as follows. Let $H$ be an antiderivative of $h$, i.e., $H' = h$. Apparently, $H(r) = \frac{n-1}2\,\ln\left(\frac r{\sinh r}\right)$. Then by applying Ito's formula we have
\beaa
  && \int_0^T h(R_t) d\tW_t 
  = H(R_T) - H(0) - \int_0^T \left[ \frac{h'(R_t)}2 + \frac{n-1}2 \frac{h(R_t)}{R_t} \right] dt.
\eeaa
It follows that the exponent of the exponential term in \eqref{eqn:f-exp-tilde} becomes
\beaa
  && -\int_0^T h(R_t) d\tW_t - \frac12\int_0^T h^2(R_t) dt \\
  &=& - H(R_T) + H(0) + \int_0^T \left[ \frac{h'(R_t)}2 + \frac{n-1}2\,\frac{h(R_t)}{R_t} - \frac{h^2(R_t)}2 \right] dt \\
  &=& \ln \left[ \frac{\sinh(R_T)}{R_T} \right]^{\frac{n-1}2} - \frac{(n-1)^2}8 T - \frac{(n-1)(n-3)}8 \int_0^T \left[\frac1{\sinh^2(R_t)} - \frac1{R_t^2} \right] dt.
\eeaa
Hence, we have
\beaa
  && \E[f(R_T)]
  = e^{- \frac{(n-1)^2T}8} \tE\left[ f(R_T) \left\{ \frac{\sinh(R_T)}{R_T} \right\}^{\frac{n-1}2} e^{- \frac{(n-1)(n-3)}8 \int_0^T \left[\frac1{\sinh^2(R_t)} - \frac1{R_t^2} \right] dt} \right].
\eeaa
In particular, when $n=3$, the last expression has a much simpler form:
\beaa
  && \E[f(R_T)] = e^{- \frac{T}2} \tE\left[ f(R_T) \frac{\sinh(R_T)}{R_T}  \right].
\eeaa
Finally, since $R_t$ in $\tP$ measure is a Bessel process of order $n$, we end up with
\beaa
  && \int_{S^{n-1}}\int_0^\infty f(r) p(T,r) \sinh^{n-1} r dr d\omega = \E[f(R_T)] \\
  &=& e^{- \frac{(n-1)^2T}8} \tE\left[ f(R_T) \left\{ \frac{\sinh(R_T)}{R_T} \right\}^{\frac{n-1}2} e^{- \frac{(n-1)(n-3)}8 \int_0^T \left[\frac1{\sinh^2(R_t)} - \frac1{R_t^2} \right] dt} \right] \\
  &=& e^{- \frac{(n-1)^2T}8} \frac{\Gamma\left(\frac n2\right)}{2\pi^{\frac n2}} \int_{S^{n-1}}\int_0^\infty f(r) \left(\frac{\sinh r}{r} \right)^{\frac{n-1}2} \times \\
  && \qquad \tE_{r} \left[ e^{- \frac{(n-1)(n-3)}8 \int_0^T \left[\frac1{\sinh^2(R_t)} - \frac1{R_t^2} \right] dt} \right]
  \frac{2 r^{n-1} e^{-\frac{r^2}{2T}}}{(2T)^{\frac n2}\Gamma\left(\frac n2\right)}  dr d\omega  \\
  &=& e^{- \frac{(n-1)^2T}8} \int_{S^{n-1}}\int_0^\infty f(r) \left(\frac{r}{\sinh r} \right)^{\frac{n-1}2} \frac{e^{-\frac{r^2}{2T}}}{(2\pi T)^{\frac n2}} \times \\
  && \qquad \tE_{r} \left[ e^{- \frac{(n-1)(n-3)}8 \int_0^T \left[\frac1{\sinh^2(R_t)} - \frac1{R_t^2} \right] dt} \right] \sinh^{n-1} r dr d\omega,
\eeaa
where in passing to the penultimate equality we used the probability density $p_B$ of the Bessel process $R_t$ given by
\[
p_B(t,r) = \frac{2 r^{n-1} e^{-\frac{r^2}{2t}}}{(2t)^{\frac n2}\Gamma\left(\frac n2\right)}
\] 
which satisfies the Fokker-Planck equation
\[
\p_t u = \frac12 \p_r^2 u - \frac{n-1}2 \p_r\left( \frac ur\right)
\]
with initial condition $u(r,0) = \delta(r)$, the Dirac delta function centered at $0$. 
Also note that the normalizing constant $\frac{2\pi^{\frac n2}}{\Gamma\left(\frac n2\right)}$ comes from the volume of $S^{n-1}$.
Thus, we obtain the bridge representation \eqref{eqn:bessel-bridge} for the transition density of hyperbolic Brownian motion.
\end{proof}
\begin{remark} $\mbox{ }$ \\
Note that when $n=3$ the representation \eqref{eqn:bessel-bridge} reduces to
\[
  p_{\H^3}(T,z,w) = e^{- \frac{T}2} \frac{r}{\sinh r} \frac{e^{-\frac{r^2}{2T}}}{(2\pi T)^{\frac 32}}
\]
which coincides with the closed form expression \eqref{eqn:hyp-ht-kern-3d}. However, for $n=2$, \eqref{eqn:bessel-bridge} reads
\[
  p_{\H^2}(T,z,w) = e^{- \frac{T}8} \sqrt{\frac{r}{\sinh r}} \frac{e^{-\frac{r^2}{2T}}}{2\pi T} \tE_{r} \left[ e^{\frac18 \int_0^T \left\{\frac1{\sinh^2(R_t)} - \frac1{R_t^2} \right\} dt} \right].
\]
Notice that
\ben
\item This expression is different from the McKean kernel \eqref{eqn:mckean} or the Gruet's formula in the sense that a) the power of $2\pi T$ is in the correct dimension ($\frac n2 = 1$) and b) the ``Gaussian" term $e^{-\frac{r^2}{2T}}$ is factored outfront naturally.

\item The integrand in the exponential term, i.e., the function $\phi(x) := \frac1{\sinh^2 x}- \frac1{x^2}$ is increasing in $[0,\infty)$ with $\lim_{x\to0^+} \phi(x) = -\frac13$ and $\lim_{x\to\infty}\phi(x) = 0$. Therefore, $\phi$ is bounded above by 0 and below by $-\frac13$.
\een
\end{remark}
As applications of the bridge representation \eqref{eqn:bessel-bridge}, a series expansion and an asymptotic expansion in small time for the hyperbolic heat kernel are almost straightforward. For notational simplicity, hereafter in this subsection we shall denote by 
\[
g(r) = - \frac{(n-1)(n-3)}8 \left[ \frac1{\sinh^2(r)} - \frac1{r^2} \right].
\]
Note that $g$ is strictly decreasing and $|g(r)| \leq \frac{(n-1)(n-3)}{24}$ for all $r > 0$.
\begin{cor} \label{cor:series}
The hyperbolic heat kernel $p_{\H^n}$ has the following series expansion
\bea
&& p_{\H^n}(T,z,w) \label{eqn:series-exp} \\
&=& e^{- \frac{(n-1)^2T}8} \left(\frac{r}{\sinh r} \right)^{\frac{n-1}2} \frac{e^{-\frac{r^2}{2T}}}{(2\pi T)^{\frac n2}} e^{\int_0^T g(r_t) dt} \, \sum_{k=0}^\infty \frac{T^k}{k!} \tE_r\left[\left(\int_0^1 g(R_{Ts}) - g(r_{Ts}) ds \right)^k\right], \nonumber
\eea
where $r_t$, for $t \in [0,T]$, is defined by
\bea
r_t = g^{-1} \left(\tE_r[g(R_t)] \right).  \label{eqn:rt} 
\eea
In other words, $g(r_t)$ is an unbiased estimator for $g(R_t)$ in the Bessel bridge measure.
\end{cor}
\begin{proof}
It suffices to deal with the conditional expectation term in \eqref{eqn:bessel-bridge}. 
\beaa
&& \tE_{r} \left[ e^{- \frac{(n-1)(n-3)}8 \int_0^T \left\{\frac1{\sinh^2(R_t)} - \frac1{R_t^2} \right\} dt} \right] \\
&=& e^{\int_0^T g(r_t) dt} \tE_{r} \left[ e^{\int_0^T \left\{g(R_t) - g(r_t) \right\} dt} \right] \\
&=& e^{\int_0^T g(r_t) dt} \, \tE_{r} \left[ \sum_{k=0}^\infty \frac1{k!} \left(\int_0^T \left\{g(R_t) - g(r_t)\right\} dt \right)^k  \right] \\
&=& e^{\int_0^T g(r_t) dt} \, \sum_{k=0}^\infty \frac1{k!} \, \tE_{r} \left[\left(\int_0^T \left\{g(R_t) - g(r_t)\right\} dt \right)^k  \right]
\eeaa
by dominating convergence theorem since the random variable $\int_0^T \left\{g(R_t) - g(r_t) \right\} dt$ is bounded. In fact, 
\[
\left|\int_0^T \left\{g(R_t) - g(r_t) \right\} dt \right| \leq \frac{(n-1)(n-3)}{12}T \qquad \mbox{almost surely.}
\]
Finally, by making the change of variable $t = Ts$ we obtain the series expansion \eqref{eqn:series-exp}.
\end{proof}
We remark that in fact we have the freedom of selecting the deterministic path $r_t$ in the series expansion \eqref{eqn:series-exp}. We choose the path as such since it serves as a first order ``unbiased estimator" in the small time asymptotic expansion in the corollary that follows. 

\begin{cor} \label{cor:small-time}
As $T \to 0^+$, the hyperbolic heat kernel $p_{\H^n}$ has the following small time asymptotic expansion up to second order
\bea
&& p_{\H^n}(T,z,w) \label{eqn:small-time-exp} \\
&=& e^{- \frac{(n-1)^2T}8} \left(\frac{r}{\sinh r} \right)^{\frac{n-1}2} \frac{e^{-\frac{r^2}{2T}}}{(2\pi T)^{\frac n2}} e^{\int_0^T g(r_t) dt} \, \left\{ 1 + O(T^2) \right\}, \nonumber
\eea
where $r_t$ is given in \eqref{eqn:rt}.
\end{cor}
\begin{proof}
Consider the infinite series on the right hand side of \eqref{eqn:series-exp},
\beaa
&& \sum_{k=0}^\infty \frac{T^k}{k!} \tE_r\left[\left(\int_0^1 g(R_{Ts}) - g(r_{Ts}) ds \right)^k\right] \\
&=& 1 + T \, \tE_{r} \left[\int_0^1 \left\{g(R_{Tu}) - g(r_{Tu}) \right\} du \right] + O(T^2) \\
&=& 1 + O(T^2)
\eeaa
by the definition of the path $r_t$. 
\end{proof}
Note that if we choose a different path $r_t$ than \eqref{eqn:rt}, then the asymptotic expansion in \eqref{eqn:small-time-exp} is of order $T$ only. 

Lastly, by na\"ively choosing $r_t$ as the straight line connecting 0 and $r$ as well as the unbiased estimator \eqref{eqn:rt}, in Figure \ref{fig:fig1} we illustrate numerically the accuracy of the asymptotic expansion \eqref{eqn:small-time-exp}, compared with the Gruet's formula. As shown in the plots, the unbiased estimator does a pretty decent job; whereas the straight line approximation is off for high dimensions. 
\begin{figure}[ht!]
\begin{center}
\includegraphics[width=10cm, height=10cm]{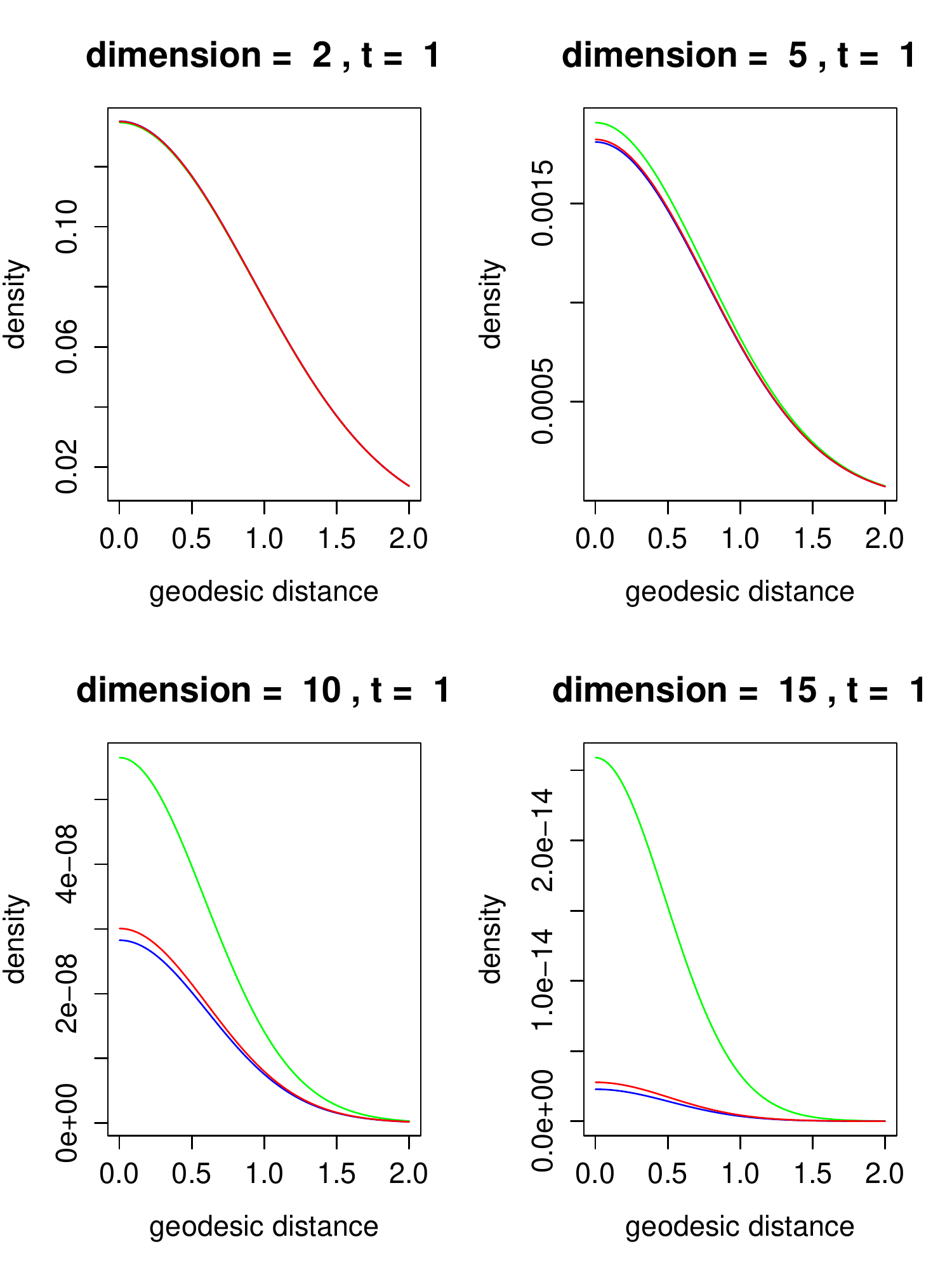} 
\end{center}
\caption{Plots of hyperbolic heat kernel at time 1 in various dimensions. Approximation of $r_t$ in \eqref{eqn:small-time-exp} by straight line in green, by the unbiasd estimator \eqref{eqn:rt} in blue. Gruet's formula in red.}
\label{fig:fig1}
\end{figure}

\subsection{Transition density of hyperbolic Bessel process}
The radial part $R_t$ of hyperbolic Brownian motion satisfying \eqref{eqn: R-process} is also referred to as the hyperbolic Bessel process. Hyperbolic Bessel processes and the calculations of their related moments are extensively explored in recent papers \cite{j-w} and \cite{pyc-zak}. 
By the same token as in Theorem \ref{thm:bessel-bridge}, we may as well derive a Bessel bridge representation for the hyperbolic Bessel process. The advantage of the bridge representation is that the expression is consistent across dimensions. However, formulas given in \cite{pyc-zak} (see Theorem 3.3), obtained by applying the Millson's recursion formula, become more and more intractable as dimension goes higher. 

\begin{theorem} \label{thm:hyperbolic-bessel}
The transition density $p_{HB}(T,x,y)$ of the hyperbolic Bessel process $R_t$ of order $n$ from $x$ to $y$ has the following Bessel bridge representation. For $T > 0$ and $x \geq 0$, $y > 0$, 
\bea
&& p_{HB}(T,x,y) \label{eqn:hyperbolic-bessel} \\
&=& e^{- \frac{(n-1)^2T}8} \left\{ \frac{\sinh(y)}{\sinh(x)} \right\}^{\frac{n-1}2} \tE_x\left[ \left. e^{- \frac{(n-1)(n-3)}8 \int_0^T \left[\frac1{\sinh^2(R_t)} - \frac1{R_t^2} \right] dt} \right| R_T = y \right] \times \nonumber \\
&& \quad \frac{e^{-\frac{x^2 + y^2}{2T}}}{\sqrt{2\pi T}}\left(\frac{xy}T\right)^{\frac{n-1}2} \frac1{2^{\frac{n-1}2} \Gamma(\frac{n-1}2)}\int_0
^\pi e^{\frac{xy}T \cos(\xi)} \sin^{n - 2}(\xi) d\xi. \nonumber
\eea
\end{theorem}
\begin{proof}
As in the proof of Theorem \ref{thm:bessel-bridge}, for any bounded measurable function $f$, the expectation of $f(R_t)$ conditioned on $R_0$ can be written as 
\beaa
&& \E[f(R_T)|R_0] \\
&=& e^{- \frac{(n-1)^2T}8} \tE_{R_0}\left[ f(R_T) \left\{ \frac{\sinh(R_T)}{R_T} \frac{R_0}{\sinh(R_0)} \right\}^{\frac{n-1}2} e^{- \frac{(n-1)(n-3)}8 \int_0^T \left[\frac1{\sinh^2(R_t)} - \frac1{R_t^2} \right] dt} \right],
\eeaa
where $\tE[\cdot]$ is the expectation in the $\tP$-measure defined in \eqref{eqn:RN-r}, under which $R_t$ is a Bessel process of order $n$. Recall that the transition density $p_B(t,x,y)$ of Bessel process of order $n$ from $x$ to $y$ in time $t$ is given by
\[
p_B(t,x,y) = \left\{\begin{array}{ll}
\frac1t \left(\frac yx\right)^\nu y e^{-\frac{x^2 + y^2}{2t}} I_\nu \left(\frac{xy}t \right) & \mbox{ if } x \neq 0; \\
& \\
\frac{2 y^{n-1} e^{-\frac{y^2}{2t}}}{(2t)^{\frac n2}\Gamma\left(\frac n2\right)} & \mbox{ if } x = 0.
\end{array}\right.
\]
where $\nu = \frac n2 - 1$. Hence, the transition density $p_{HB}(t, x, y)$ (from $x$ to $y$ in time $t$) for the hyperbolic Bessel process (i.e., $R_t$ in the $\P$-measure) has the representation
\beaa
&& p_{HB}(T, x, y) \\
&=& e^{- \frac{(n-1)^2T}8} \left\{ \frac{\sinh(y)}y \frac{x}{\sinh(x)} \right\}^{\frac{n-1}2} \tE_x \left[\left. e^{- \frac{(n-1)(n-3)}8 \int_0^T \left[\frac1{\sinh^2(R_t)} - \frac1{R_t^2} \right] dt} \right|R_T = y\right] \times \\
&& \quad \frac1T \left(\frac yx\right)^\nu y e^{-\frac{x^2 + y^2}{2T}} I_\nu \left(\frac{xy}T \right) \\
&=& e^{- \frac{(n-1)^2T}8} \left\{ \frac{\sinh(y)}{\sinh(x)} \right\}^{\frac{n-1}2} \tE_x \left[\left. e^{- \frac{(n-1)(n-3)}8 \int_0^T \left[\frac1{\sinh^2(R_t)} - \frac1{R_t^2} \right] dt} \right| R_T = y\right] \times \\
&& \quad \frac{e^{-\frac{x^2 + y^2}{2T}}}{\sqrt{2\pi T}}\left(\frac{xy}T\right)^{\frac{n-1}2} \frac1{2^{\frac{n-1}2} \Gamma(\frac{n-1}2)}\int_0
^\pi e^{\frac{xy}T \cos(\xi)} \sin^{n - 2}(\xi) d\xi,
\eeaa
where in the last equality we used the following integral representation for the modified Bessel function $I_\nu$
\[
I_{\nu}(z) = \frac{z^\nu}{2^{\nu}\sqrt\pi\,\Gamma(\nu + \frac12)} \int_0^\pi e^{z \cos(\xi)} \sin^{2\nu}(\xi) d\xi.
\]
\end{proof}
In particular, when $n=3$, \eqref{eqn:hyperbolic-bessel} can be expressed in elementary functions as 
\beaa
&& p_{HB}(t,x,y) \\
&=& e^{- \frac t2}\, \frac{\sinh(y)}{\sinh(x)} \, \frac{e^{-\frac{x^2 + y^2}{2t}}}{\sqrt{2\pi t}}\, \frac{xy}t \, \frac12 \int_0^\pi e^{\frac{xy}t \cos(\xi)} \sin(\xi) d\xi \\
&=& \frac{e^{- \frac t2}}{\sqrt{2\pi t}} \frac{\sinh(y)}{\sinh(x)} e^{-\frac{x^2 + y^2}{2t}} \left(e^{\frac{xy}t} - e^{-\frac{xy}t}\right) \\
&=& \frac{e^{- \frac t2}}{\sqrt{2\pi t}} \frac{\sinh(y)}{\sinh(x)} \left(e^{-\frac{(x - y)^2}{2t}} - e^{-\frac{(x + y)^2}{2t}}\right)
\eeaa
which coincides with the formula in \cite{pyc-zak}. We summarize the result in the following corollary. 

\begin{cor}
The transition density $p_{HB}$ of the hyperbolic Bessel process $R_t$ of order $3$ has the following closed form expression. For $t > 0$ and $x \geq 0$, $y > 0$, 
\beaa
&& p_{HB}(t,x,y) = \frac{e^{- \frac t2}}{\sqrt{2\pi t}} \frac{\sinh(y)}{\sinh(x)} \left(e^{-\frac{(x - y)^2}{2t}} - e^{-\frac{(x + y)^2}{2t}}\right).
\eeaa
\end{cor}

Notice that, since the conditional expectation term in \eqref{eqn:hyperbolic-bessel} is exactly the same as the one in \eqref{eqn:bessel-bridge}, one can easily derive series and small time asymptotic expansions for the transition density of hyperbolic Bessel process, similarly as the ones in Corollaries \ref{cor:series} and \ref{cor:small-time}. For example, we have
\begin{cor} 
As $T \to 0^+$, the transition density $p_{HB}$ of hyperbolic Bessel process has the following small time asymptotic expansion up to second order
\beaa
&& p_{HB}(T,x,y) \\
&=& e^{- \frac{(n-1)^2T}8} \left\{ \frac{\sinh(y)}{\sinh(x)} \right\}^{\frac{n-1}2} e^{- \frac{(n-1)(n-3)}8 \int_0^T \left[\frac1{\sinh^2(r_t)} - \frac1{r_t^2} \right] dt}  \times \nonumber \\
&& \quad \frac{e^{-\frac{x^2 + y^2}{2T}}}{\sqrt{2\pi T}}\left(\frac{xy}T\right)^{\frac{n-1}2} \frac1{2^{\frac{n-1}2} \Gamma(\frac{n-1}2)}\int_0
^\pi e^{\frac{xy}T \cos(\xi)} \sin^{n - 2}(\xi) d\xi \; \left\{ 1 + O(T^2) \right\}. \nonumber
\eeaa
where $r_t$ is given in \eqref{eqn:rt}.
\end{cor}

%%%%%%%%%%%%%%%%%%%%%%%%%%%%%%%%%%%%%%%%%%%%%%%%%%%%%%%%%%%%%%%%%%%%%%%%%%%%%%%%%%%%%%%%
%
%  Section: Bridge representation in radially symmetric spaces
%
%%%%%%%%%%%%%%%%%%%%%%%%%%%%%%%%%%%%%%%%%%%%%%%%%%%%%%%%%%%%%%%%%%%%%%%%%%%%%%%%%%%%%%%%

\section{Bridge representation in radially symmetric spaces}
Let $M$ be a radially symmetric space that is also Cartan-Hadamard. We recall that a Cartan-Hadamard manifold is a negatively curved, complete and simply connected Riemannian manifold whose exponential map at any given point defines a global diffeomorphism. Consequently, the geodesic polar coordinates at pole are globally defined for such manifolds. We refer the reader to \cite{eltonbook} and the references therein for more detailed discussions. In geodesic polar coordinates, the Riemann metric $ds^2$ and the Laplace-Beltrami operator $\Delta_M$ on $M$ can be written respectively as
\bea
&& ds^2 = dr^2 + G^2(r) d\theta^2, \label{eqn:radial-sym-metric} \\
&& \Delta_M = \p_r^2 + (n-1) \frac{G'(r)}{G(r)} \p_r + \frac1{G^2(r)} \Delta_{S^{n-1}}, \label{eqn:radial-sym-LBop}
\eea
where $r$ is the geodesic distance and, as before, $d\theta^2$ denotes the standard Riemann metric over the unit sphere $S^{n-1}$. The radial function $G$ is nonnegative and satisfies $G(0) = 0$ and $G'(0) = 1$. Conceivably due to symmetry, heat kernel on such spaces has analogous Bessel bridge representation as for the hyperbolic space with minor modifications. We present the representation in the following theorem but omit its proof since it is almost identical with the proof of Theorem \ref{thm:bessel-bridge}.  
\begin{theorem}(Bessel bridge representation in radially symmetric space) \label{thm:bessel-bridge-radial-sym} \\ 
Let $M$ be an $n$-dimensional Cartan-Hadamard radial symmetric space with Riemann metric and Laplace-Beltrami operator given by \eqref{eqn:radial-sym-metric} and \eqref{eqn:radial-sym-LBop} respectively at its pole $z \in M$. Further assume that $G$ satifies the regularity condition $\left|\frac{d}{dr}\ln\frac{G(r)}r \right| \leq C$ for some $C > 0$. Then, for $w \in M$, the heat kernel $p_M(T,z,w)$ has the following representation
\bea
&& p_M(T,z,w) \nonumber \\
&=& \left\{\frac{r}{G(r)} \right\}^{\frac{n-1}2} \frac{e^{-\frac{r^2}{2T}}}{(2\pi T)^{\frac n2}} \E_r \left[ e^{\int_0^T \left(\frac{(n-1)(n-3)}8 \left\{\frac1{R_t^2} - \left(\frac{G'(R_t)}{G(R_t)} \right)^2 \right\} - \frac{n-1}4 \frac{G''(R_t)}{G(R_t)} \right)dt} \right], \label{eqn:bessel-bridge-radial-sym}
\eea
where $r = r(z,w)$ is the geodesic distance between $z$ and $w$. $\E_r[\cdot]$ denotes the Bessel bridge measure, i.e., the conditional expectation $\E[\cdot|R_T = r]$.  
\end{theorem}
\begin{remark}
Similarly, in three dimensional case, $n=3$, the representation \eqref{eqn:bessel-bridge-radial-sym} has the following slightly simpler form
\[
p(T,z,w)
= \frac{r}{G(r)} \frac{e^{-\frac{r^2}{2T}}}{(2\pi T)^{\frac 32}} \tE_{r} \left[ e^{-\frac12 \int_0^T  \frac{G''(R_t)}{G(R_t)} dt} \right],
\]
where again $\E_r[\cdot]$ denotes the expectation under Bessel bridge measure. Apparently, it recovers $p_{\H^3}$ in \eqref{eqn:hyp-ht-kern-3d} by setting $G(r) = \sinh r$. 
\end{remark}
Finally, a direct application of \eqref{eqn:bessel-bridge-radial-sym} is the following expansion in small time of the heat kernel on radially symmetry space:
\begin{cor}
As $T \to 0^+$,
\beaa
&& p_M(T,z,w) = \left\{\frac{r}{G(r)} \right\}^{\frac{n-1}2} \frac{e^{-\frac{r^2}{2T}}}{(2\pi T)^{\frac n2}} \times \\
&& \quad \left\{ 1 + \int_0^T \left(\frac{(n-1)(n-3)}8 \E_r\left[ \frac1{R_t^2} - \left(\frac{G'(R_t)}{G(R_t)} \right)^2 \right] - \frac{n-1}4 \E_r\left[\frac{G''(R_t)}{G(R_t)}\right] \right)dt + O(T^2)\right\},
\eeaa
where $R_t$ is a Bessel bridge of order $n$ connecting 0 and $r$ in time $T$.
\end{cor}

%%%%%%%%%%%%%%%%%%%%%%%%%%%%%%%%%%%%%%%%%%%%%%%%%%%%%%%%%%%%%%%%%%%%%%%%%%%%%%%%%%%%%%%%
%
%  Section: Acknowledgement
%
%%%%%%%%%%%%%%%%%%%%%%%%%%%%%%%%%%%%%%%%%%%%%%%%%%%%%%%%%%%%%%%%%%%%%%%%%%%%%%%%%%%%%%%%

\section*{Acknowledgement}
We thank the referee for a careful reading and comments of the manuscript. The authors are partially supported by the Natural Science Foundation of China grant 11601018. XC is also partially supported by the Natural Science Foundation of China grant 11471051.

%%%%%%%%%%%%%%%%%%%%%%%%%%%%%%%%%%%%%%%%%%%%%%%%%%%%%%%%%%%%%%%%%%%%%%%%%%%%%%%%%%%%%%%%
%
%  Bibliography
%
%%%%%%%%%%%%%%%%%%%%%%%%%%%%%%%%%%%%%%%%%%%%%%%%%%%%%%%%%%%%%%%%%%%%%%%%%%%%%%%%%%%%%%%%

\end{document}